\documentclass[11pt]{article}
\usepackage{amscd, amsmath, amssymb, amsthm}
\usepackage[all,cmtip]{xy}
\usepackage[pagebackref]{hyperref}

\title{The integral Hodge conjecture for 3-folds
of Kodaira dimension zero}
\author{Burt Totaro}
\date{  }

\def\Z{\text{\bf Z}}
\def\Q{\text{\bf Q}}
\def\R{\text{\bf R}}
\def\C{\text{\bf C}}
\def\P{\text{\bf P}}

\def\arrow{\rightarrow}
\def\inj{\hookrightarrow}
\def\imp{\Rightarrow}

\def\Hom{\text{Hom}}

\def\van{\text{van}}
\def\im{\text{im}}
\def\rank{\text{rank}}
\def\Gal{\text{Gal}}

\def\o{\underline{\Omega}}
\def\Aut{\text{Aut}}
\def\Pic{\text{Pic}}
\def\et{\text{et}}
\def\Br{\text{Br}}
\def\X{\widehat{X}}

\begin{document}
\maketitle
\newtheorem{theorem}{Theorem}[section]
\newtheorem{proposition}[theorem]{Proposition}
\newtheorem{corollary}[theorem]{Corollary}
\newtheorem{lemma}[theorem]{Lemma}

\theoremstyle{definition}
\newtheorem{definition}[theorem]{Definition}
\newtheorem{example}[theorem]{Example}

\theoremstyle{remark}
\newtheorem{remark}[theorem]{Remark}

The Hodge conjecture is true for all smooth complex projective 3-folds,
by the Lefschetz $(1,1)$ theorem and
the hard Lefschetz theorem \cite[p.~164]{GH}.
The integral Hodge conjecture is a stronger statement which 
fails for some 3-folds, in fact for some
smooth hypersurfaces in $\P^4$, by Koll\'ar \cite{KollarTrento}.
Voisin made a dramatic
advance by proving the integral Hodge conjecture for all uniruled 
3-folds (or, equivalently, all 3-folds with Kodaira dimension
$-\infty$) and all 3-folds $X$ with trivial canonical bundle $K_X$
and first Betti number zero \cite{Voisinint}. Also, Grabowski proved
the integral Hodge conjecture for abelian 3-folds
\cite[Corollary 3.1.9]{Grabowski}.

In this paper, we prove the integral
Hodge conjecture for all smooth projective 3-folds $X$
of Kodaira dimension zero
with $h^0(X,K_X)>0$ (hence equal to 1).
This generalizes the results of Voisin
and Grabowski in two directions.
First, it includes all smooth projective
3-folds with trivial canonical bundle,
not necessarily with first Betti number zero.
For example, the integral Hodge conjecture holds for quotients
of an abelian 3-fold by a free action of a finite group preserving
a volume form, and for
volume-preserving quotients of a K3 surface times an elliptic curve.
Second,
our result includes any smooth projective 3-fold whose minimal model
is a possibly singular variety with trivial canonical bundle.

In contrast, Benoist
and Ottem showed that the integral Hodge conjecture can fail
for 3-folds of any Kodaira dimension $\geq 0$. In particular,
it can fail for an Enriques surface times an elliptic curve;
in that case, $X$ has Kodaira dimension zero, and in fact
the canonical bundle is torsion of order 2
\cite{BO}. So our positive result is sharp in a strong
sense.

The proof here covers all cases (including abelian 3-folds) in a unified
way, building on the arguments of Voisin and H\"oring-Voisin \cite{HV}.
In order to show that
a given homology class is represented by an algebraic 1-cycle on $X$,
we consider a family of surfaces of high degree in a minimal model
of $X$.
The 1-cycle we want cannot be found on most surfaces in the family,
but it will appear on some surface in the family. This uses an analysis
of Noether-Lefschetz loci, which depends on the assumption
that $h^0(X,K_X)>0$.

As an application of what we know about the integral Hodge conjecture,
we prove the integral Tate conjecture for all rationally connected 3-folds
and all 3-folds of Kodaira dimension zero with $h^0(X,K_X)>0$
in characteristic zero
(Theorem \ref{tatezero}). Finally, we prove the integral
Tate conjecture for abelian 3-folds in any characteristic
(Theorem \ref{abeliantate}).

I thank Yuri Tschinkel for a good question.
This work was supported by National Science Foundation
grant DMS-1701237, and by grant DMS-1440140
while the author was in residence at the
Mathematical Sciences Research Institute in Berkeley, California,
during the Spring 2019 semester.

\section{Notation}

The {\it integral Hodge conjecture }for
a smooth complex projective variety $X$ asserts that every element
of $H^{2i}(X,\Z)$ whose image in $H^{2i}(X,\C)$ is of type $(i,i)$
is the class of an algebraic cycle of codimension $i$,
that is, a $\Z$-linear combination
of subvarieties of $X$. The {\it Hodge conjecture }is
the analogous statement for rational cohomology and algebraic cycles
with rational coefficients.
The {\it integral Tate conjecture }for
a smooth projective variety $X$ over the separable closure $F$ of a finitely
generated field says: for $k$ a finitely generated
field of definition of $X$ whose separable closure is $F$
and $l$ a prime number invertible
in $k$, every element of $H^{2i}(X_F,\Z_l(i))$
fixed by some open subgroup of $\Gal(F/k)$ is the class of an algebraic
cycle over $F$ with $\Z_l$ coefficients. Although it does not hold
for all varieties, this version of the integral
Tate conjecture holds
in more cases than the analogous statement
over the finitely generated field $k$ \cite[section 1]{Totaro}.
The {\it Tate conjecture }is the analogous statement with
$\Q_l$ coefficients.

On a normal variety $Y$, we use a natural generalization of the vector bundle
of differential forms on a smooth variety, the sheaf
$\Omega^{[j]}_Y$ of {\it reflexive differentials}:
$$\Omega^{[j]}_Y:=(\Omega^j_Y)^{**}=i_*\Omega^j_U,$$
where $i\colon U\arrow Y$ is the inclusion of the smooth locus.

For a vector space $V$, $P(V)$ denotes the space
of hyperplanes in $V$.

\section{Examples}

In this section, we discuss some examples of 3-folds
satisfying our assumptions, and how our proof works in various
cases. One interesting point is the following dichotomy
among 3-folds satisfying our assumptions. 
This dichotomy will not be used in the rest of the paper,
but the proof of Lemma \ref{structure}
develops some basic properties
of these 3-folds that will be used.

\begin{lemma}
\label{structure}
Let $X$ be a smooth projective complex 3-fold of Kodaira
dimension zero with $h^0(X,K_X)>0$ (hence equal to 1).
Let $Y$ be a minimal model of $X$. (Here $Y$ is a terminal
3-fold with $K_Y$ trivial.) Then either $H^1(X,O)=H^1(Y,O)$
is zero or $Y$ is smooth (or both).
\end{lemma}

Note that the integral Hodge conjecture for smooth projective
3-folds is a birationally
invariant property \cite[Lemma 15]{Voisinsome}. Therefore,
to prove Theorem \ref{minimal} (the integral Hodge conjecture for $X$
as above), we could assume that $H^1(X,O)=0$ or else
that $K_X$ is trivial (although we will not in fact divide
up the proof of Theorem \ref{minimal} that way).
The case with $H^1(X,O)=0$ follows
from work of H\"oring and Voisin \cite[Proposition 3.18]{HV}
together with
a relatively easy analysis of singularities below
(Lemma \ref{sing}). (To give examples of such 3-folds:
there are many terminal
hypersurface singularities in dimension 3, such as
any isolated singularity of the form $xy+f(z,w)=0$ for some power series $f$
\cite[Definition 3.1, Corollary 3.12]{Reid}, and $X$ could
be any resolution of a terminal quintic 3-fold in $\P^4$.)

The case of smooth projective 3-folds $X$
with $K_X$ trivial (but $H^1(X,O)$ typically not zero) is harder,
and requires a thorough reworking of H\"oring and Voisin's arguments.
We discuss examples of such varieties after the following proof.

\begin{proof}
(Lemma \ref{structure})
By Mori, there is a minimal model $Y$ of $X$ \cite[2.14]{KM}.
That is, $Y$ is a terminal
3-fold whose canonical divisor $K_Y$ is nef, with a birational
map from $X$ to $Y$.
Terminal varieties are smooth in codimension 2,
and so $Y$ is smooth
outside finitely many points. Since $X$ has Kodaira dimension zero,
the Weil divisor class $K_Y$ is torsion
by the abundance theorem for 3-folds,
proved by Kawamata and Miyaoka \cite{Kawamata}.
Since $h^0(Y,K_Y)=h^0(X,K_X)>0$,
$K_Y$ is trivial (and hence $h^0(X,K_X)=h^0(Y,K_Y)=1$).

Since $K_Y$ is linearly equivalent to zero,
$K_Y$ is in particular a Cartier divisor. 
Also, since $Y$ is terminal, it has rational singularities
\cite[3.8]{Reid}; in particular, it is Cohen-Macaulay.
So the line bundle $K_Y$ is the dualizing sheaf of $Y$.
Since $Y$ is a terminal 3-fold $Y$ with $K_Y$ Cartier,
it has only hypersurface
(hence lci) singularities, by Reid \cite[Theorem 3.2]{Reid}.
Let $S$ be a smooth ample Cartier divisor in $Y$; then $S$
is contained in the smooth locus of $Y$.
By Goresky and MacPherson,
$i_*\colon H_2(S,\Z)\arrow H_2(Y,\Z)$ is surjective, using
that $Y$ has only lci singularities
\cite[p.~24]{Fultontopology}. 

For any scheme $Y$ of finite type over the complex numbers,
du Bois constructed a canonical object $\o^*_Y$ in the filtered
derived category of $Y$, isomorphic to the constant sheaf $\C_Y$
in the usual derived category $D(Y)$ \cite{DuBois}. For $Y$ smooth,
this is simply the de Rham complex. Write $\o^j_Y$ in $D(Y)$ for the $j$th
graded piece of $\o^*_Y$ with respect to the given filtration, shifted
$j$ steps to the left;
for $Y$ smooth, this is the sheaf $\Omega^j_Y$ in degree zero.
For $Y$ proper
over $\C$, the resulting spectral sequence 
$$E_1^{pq}=H^q(Y,\o^p_Y)\imp H^{p+q}(Y,\C)$$
degenerates at $E_1$ \cite[Theorem 4.5]{DuBois}.
The associated filtration on $H^*(Y,\C)$
is the Hodge filtration defined by Deligne.

The objects $\o^j_Y$ need not be sheaves, even in our very special
situation, where $Y$ has terminal 3-fold hypersurface singularities.
In particular, Steenbrink showed that $\o^1_Y$ has nonzero cohomology
in degree $1$ (not just degree 0) for any isolated rational complete
intersection 3-fold singularity other than a node or a smooth point
\cite[p.~1374]{SteenbrinkDuBois}.

In our case, because $H_2(S,\Z)\arrow H_2(Y,\Z)$ is surjective,
the pullback $H^2(Y,\Q)\arrow H^2(S,\Q)$ is injective.
By strict compatibility of pullback maps with the weight
filtration, it follows that the mixed Hodge structure
on $H_2(Y,\Q)$ is pure of weight $-2$ \cite{Delignepoids},

So $H^2(Y,\Q)$ is pure of weight 2. By the discussion above,
the graded pieces of the Hodge filtration on $H^2(Y,\C)$
are $H^2(Y,\o^0_Y)$, $H^1(Y,\o^1_Y)$, and $H^0(Y,\o^2_Y)$.
Since $Y$ is terminal (log canonical would be enough),
it is du Bois, which means
that $\o^0_Y\cong O_Y$ \cite{KK}.

Let $TY=(\Omega^1_Y)^*$, which is a reflexive sheaf on $Y$.
The Lie algebra of the automorphism group of $Y$
is $H^0(Y,TY)$. We have $TY\cong \Omega^{[2]}_Y\otimes K_Y^*
\cong \Omega^{[2]}_Y$, since $K_Y$ is trivial. So
$H^0(Y,TY)\cong H^0(Y,\Omega^{[2]})$. Du Bois's object
$\o^2_Y$ in the derived category of $Y$
has ${\cal H}^0(\o^2_Y)\cong \Omega^{[2]}$ since $Y$ is klt
\cite[Theorems 5.4 and 7.12]{HJ}. Since the cohomology sheaves
of $\o^2_Y$ are concentrated in degrees $\geq 0$,
it follows that $H^0(Y,\Omega^{[2]}_Y)\cong H^0(Y,\o^2_Y)$.

The polarization of $H^2(S_{t_0},\Q)$ by the intersection
form gives a canonical direct-sum decomposition of Hodge structures
\cite[Lemma 7.36]{Voisinbook}:
$$H^2(S_{t_0},\Q)=H^2(Y,\Q)\oplus H^2(S_{t_0},\Q)^{\perp}.$$
The restriction of this polarization
gives a polarization of the Hodge structure $H^2(Y,\Q)$;
this can be described
as the polarization of $H^2(Y,\Q)$ given by the ample line
bundle $H=O(S)$ on $Y$.

In particular,
the polarization of $H^2(Y,\Q)$ gives an isomorphism $H^0(Y,\o^2_Y)
\cong H^2(Y,\o^0_Y)^*\cong H^2(Y,O)^*$. 
By Serre duality and the triviality of $K_Y$,
we have $H^2(Y,O)^*\cong H^1(Y,K_Y)\cong H^1(Y,O)$.
Putting this all together,
we have $H^0(Y,TY)\cong H^1(Y,O)$.

Thus, if $H^1(X,O)$ is not zero, then the identity component
$\Aut^0(Y)$ of $\Aut(Y)$ has positive dimension. By the Barsotti-Chevalley
theorem, $\Aut^0(Y)$
is an extension of an abelian variety
by a connected linear algebraic group \cite[Theorem 8.27]{Milne}.
Any connected linear algebraic group over $\C$ is unirational
\cite[Theorem 17.93]{Milne}. Since $Y$ has Kodaira
dimension 0, it is not uniruled, and so it has no nontrivial
action of a connected linear algebraic group. We conclude
that $A:=\Aut^0(Y)$ is an abelian variety of positive dimension.

By Brion, extending work of Nishi and Matsumura,
any faithful action of an abelian variety on a normal
quasi-projective variety has finite stabilizer groups
\cite[Theorem 2]{Brion}. In our case, $A$ preserves
the singular locus of $Y$, which has dimension at most 0
because $Y$ is a terminal 3-fold. Since $A$ has positive dimension,
the singular locus of $Y$ must be empty.
\end{proof}

We conclude the section by giving examples of smooth
projective 3-folds $X$ with $K_X$ trivial
and $H^1(X,O_X)\neq 0$, beyond the obvious examples:
a K3 surface times an elliptic curve, or an abelian 3-fold.
The Beauville-Bogomolov structure theorem implies
that $X$ is a quotient
of a variety $Z$ of one of those special types by a free
action of a finite group \cite{BeauvilleCY}.
Knowing the integral Hodge conjecture for $Z$ does
not obviously imply it for $X$, which helps to motivate
this paper.

\begin{example}
An action of a finite group $G$ on a complex K3 surface $S$
is said to be {\it symplectic} if $G$ acts as the identity
on $H^0(S,K_S)\cong \C$. Mukai (completing earlier work of Nikulin)
classified the finite groups that
can act faithfully and symplectically on some K3 surface. In particular,
the abelian groups that can occur are: $\Z/a$ for $1\leq a\leq 8$,
$(\Z/2)^2$, $(\Z/2)^3$, $(\Z/2)^4$, $(\Z/3)^2$, $(\Z/4)^2$,
$\Z/2\times \Z/4$, and $\Z/2\times \Z/6$ \cite[Theorem 4.5(b)
and note added in proof]{Nikulin},
\cite[Theorem 0.6]{Mukai}.

Let $G$ be a nontrivial group on this list other than
$(\Z/2)^3$ or $(\Z/2)^4$, and let $G$ act symplectically
on a K3 surface $S$. Let $E$ be any complex elliptic curve. We can
choose an embedding of $G$ as a subgroup of $E$. Let
$X=(S\times E)/G$, where $G$ acts in the given way on $S$
and by translations on $E$. Then $X$ is a smooth projective
3-fold with $K_X$ trivial. Moreover, $H^1(X,O)$ is not zero,
because $X$ maps onto the elliptic curve $E/G$. Finally,
$X$ is not the product of a K3 surface with an elliptic curve.
So it is a new case for which Theorem \ref{minimal} proves
the integral Hodge conjecture.
\end{example}

\begin{example}
Theorem \ref{minimal} also applies to some quotients
of abelian 3-folds. For example, let $S$ be a complex abelian
surface, and let $G$ be a finite abelian group with at most 2 generators
which acts faithfully and symplectically on $S$ as an abelian surface.
Let $E$ be any elliptic curve. Choose an embedding of $G$ as a subgroup
of $E$. Let $X=(S\times E)/G$, where
$G$ acts in the given way on $S$ and by translations on $E$.
Then $X$ is a smooth projective 3-fold, $K_X$ is trivial,
and $H^1(X,O)$ is not zero, because $X$ maps onto
the elliptic curve $E/G$. Here $X$ is not an abelian 3-fold,
and so it is a new case for which this paper
proves the integral Hodge conjecture. The simplest
case is $G=\Z/2$, acting on any abelian surface $S$
by $\pm 1$.
\end{example}

\section{Terminal 3-folds}

We here analyze the homology of the exceptional divisor
of a resolution
of an isolated rational 3-fold singularity (Lemma \ref{sing}).
This will be used in proving the integral Hodge conjecture
for certain 3-folds whose minimal model is singular
(Theorem \ref{minimal}).

\begin{lemma}
\label{sing}
Let $Y$ be a complex 3-fold with isolated rational singularities.
Let $\pi\colon X\arrow Y$ be a projective birational morphism
with $X$ smooth such that $\pi$
is an isomorphism over the smooth locus of $Y$ and 
the inverse image of the singular locus of $Y$ is a divisor $D$
in $X$ with simple normal crossings. Then $H_2(D,\Z)$ is generated
by algebraic 1-cycles on $D$.
\end{lemma}

\begin{proof}
We start with the following result by Steenbrink 
\cite[Lemma 2.14]{SteenbrinkMHS}.

\begin{lemma}
\label{rational}
Let $\pi\colon X\arrow Y$ be a log resolution of an isolated
rational singularity, with exceptional divisor $D$.
Then $H^{i}(D,O)=0$ for all $i>0$.
\end{lemma}

We continue the proof of Lemma \ref{sing}.
Let $D_1,\ldots,D_r$ be the irreducible components of $D$,
which are smooth projective surfaces. Write $D_{i_0\cdots i_l}$
for an intersection $D_{i_0}\cap\cdots \cap D_{i_l}$.
We have an exact sequence
of coherent sheaves on $D$:
$$0\arrow O_D\arrow \oplus_i O_{D_i}\arrow \oplus_{i<j}O_{D_{ij}}
\arrow \oplus_{i<j<k}O_{D_{ijk}}\arrow 0.$$
Taking cohomology gives a Mayer-Vietoris spectral sequence
$$E_1^{p,q}=\oplus_{i_0<\cdots <i_p}H^q(D_{i_0\cdots i_p},O)
  \imp H^{p+q}(D,O).$$
$$\xymatrix{\oplus H^2(D_i,O)\ar[r]& 0& 0& 0\\
\oplus H^1(D_i,O)\ar[r]\ar@{-->}[rrd]&
  \oplus H^1(D_{ij},O)\ar[r]& 0& 0\\
\oplus H^0(D_i,O)\ar[r]& \oplus H^0(D_{ij},O)\ar[r]& 
  \oplus H^0(D_{ijk},O)\ar[r]& 0}$$
We have $H^2(D,O)=0$ by Lemma \ref{rational}. It follows
from the spectral sequence that each
irreducible component $D_i$ of $D$ has $H^2(D_i,O)=0$.

There is also a Mayer-Vietoris spectral sequence
for the integral homology of $D$:
$$E^1_{p,q}=\oplus_{i_0<\cdots<i_p}H_q(D_{i_0\cdots
i_p},\Z)\imp H_{p+q}(D,\Z).$$
$$\xymatrix{\oplus H_2(D_i,\Z)& \oplus H_2(D_{ij},\Z)\ar[l]& 0\ar[l] & 0\\
\oplus H_1(D_i,\Z)&
  \oplus H_1(D_{ij},\Z)\ar[l]& 0\ar[l]& 0\\
\oplus H_0(D_i,\Z)& \oplus H_0(D_{ij},\Z)\ar[l]& 
  \oplus H_0(D_{ijk},\Z)\ar[l]\ar@{-->}[llu]& 0\ar[l]}$$

Finally, we have a Mayer-Vietoris spectral sequence
converging to $H^*(D,\C)$, which can be obtained from
the integral homology spectral sequence by applying
$\Hom(\cdot,\C)$. We have a map of spectral sequences
from the one converging to $H^*(D,\C)$ to the one converging
to $H^*(D,O)$. Since $H^2(D,O)=0$, we know that the
groups $E_{\infty}^{1,1}$ and $E_{\infty}^{2,0}$ 
are zero in the spectral sequence
converging to $H^*(D,O)$. That is, the $d_1$ and $d_2$ differentials
together map onto $\oplus H^1(D_{ij},O)$ and $\oplus H^0(D_{ijk},O)$.
We will deduce that the $d_1$ and $d_2$ differentials together map
onto $\oplus H^1(D_{ij},\C)$ and $\oplus H^0(D_{ijk},\C)$.
In the first case, we are given that 
$d_1\colon \oplus H^1(D_i,O)\arrow \oplus H^1(D_{ij},O)$
is surjective, and we want to deduce that
$d_1\colon \oplus H^1(D_i,\C)\arrow \oplus H^1(D_{ij},\C)$
is surjective. That follows from $d_1\colon \oplus H^1(D_i,\C)\arrow
\oplus H^1(D_{ij},\C)$ being a morphism of Hodge structures 
of weight 1, so that $H^1(D_i,\C)=H^1(D_i,O)\oplus
\overline{H^1(D_i,O)}$ and this grading is compatible with
the differential.

A similar argument applies to $H^0$. First,
the differential $d_1\colon \oplus H^0(D_{ij},\C)\arrow \oplus H^0(D_{ijk},
\C)$ maps isomorphically to $d_1\colon \oplus H^0(D_{ij},O)
\arrow \oplus H^0(D_{ijk},O)$. Also, by the comment about
Hodge structures of weight 1, $E_2^{0,1}(\C)=\ker(\oplus H^1(D_i,\C)
\arrow \oplus H^1(D_{ij},\C))$ is the direct sum of
$E_2^{0,1}(O)=\ker(\oplus H^1(D_i,O)
\arrow \oplus H^1(D_{ij},O))$ and its conjugate, and so
$E_2^{0,1}(\C)\arrow E_2^{0,1}(O)$ is surjective. Since
$d_2\colon E_2^{0,1}(O)\arrow E_2^{2,0}(O)$ is onto and
$E_2^{0,1}(\C)\arrow E_2^{0,1}(O)$ is onto, it follows
that $d_2\colon E_2^{0,1}(\C)\arrow E_2^{2,0}(\C)$ ($=E_2^{2,0}(O)$)
is onto. That is, $E_{\infty}^{2,0}(\C)$ as well as 
$E_{\infty}^{1,1}(\C)$ are zero. Therefore $H^2(D,\C)\arrow
\oplus H^2(D_i,\C)$ is injective. (In particular, the mixed Hodge
structure on $H^2(D,\Q)$ is pure of weight 2 \cite{Delignepoids}.)
By the universal
coefficient theorem, it follows that $\oplus H_2(D_i,\Q)
\arrow H_2(D,\Q)$ is surjective.

The groups $H_0(D_{ijk},\Z)$ and $H_1(D_{ij},\Z)$ are torsion-free,
since $D_{ijk}$ is a point or empty and $D_{ij}$ is a smooth
projective curve or empty. It follows that the subgroups
$E^{\infty}_{2,0}(\Z)$ and $E^{\infty}_{1,1}(\Z)$ of these groups
are also torsion-free. Since $\oplus H_2(D_i,\Q)\arrow H_2(D,\Q)$
is surjective, those two $E^{\infty}$ groups are zero after
tensoring with the rationals, and so they are zero. Therefore,
$\oplus H_2(D_i,\Z)\arrow H_2(D,\Z)$ is surjective. Since
$H^2(D_i,O)=0$, the Lefschetz $(1,1)$ theorem gives that
the smooth projective surface $D_i$ has
$H_2(D_i,\Z)$ spanned by algebraic cycles. We deduce that
$H_2(D,\Z)$ is spanned by algebraic cycles.
\end{proof}

\section{3-folds of Kodaira dimension zero}

In this section, we begin the proof of our main
result on the integral Hodge conjecture, Theorem \ref{minimal}.
We reduce the problem
to a statement on the variation of Hodge structures
associated to a family of surfaces of high degree
in a minimal model of the 3-fold,
to be proved in the next section (Proposition \ref{hodge}).

\begin{theorem}
\label{minimal}
Let $X$ be a smooth projective complex 3-fold
of Kodaira dimension zero such that $h^0(X,K_X)>0$.
Then $X$ satisfies the integral Hodge conjecture.
\end{theorem}

\begin{proof}
Let $Y$ be a minimal model of $X$. Then $Y$ is terminal
and hence has singular set of dimension at most zero.
As in the proof of Lemma \ref{structure},
$K_Y$ is trivial (and hence $h^0(X,K_X)=h^0(Y,K_Y)=1$).

For codimension-1 cycles, the integral Hodge
conjecture always holds, by the Lefschetz $(1,1)$ theorem.
It remains to prove the integral Hodge conjecture for codimension-2
cycles on $X$. 
This is a birationally
invariant property for smooth projective varieties $X$
\cite[Lemma 15]{Voisinsome}. Therefore, we can assume
that the birational map $X\dashrightarrow Y$ is a morphism,
and that $X$ is whatever resolution of $Y$ we like. Explicitly,
we can assume that
$X\arrow Y$ is an isomorphism over
the smooth locus and that the fiber over each of the (finitely many)
singular points of $Y$ is a divisor with simple
normal crossings. (We do this in order to apply Lemma \ref{sing}.)

Let $H$ be a very ample line bundle on $Y$, and $S$ a smooth
surface in the linear system $|H|$.
As shown in the proof of Lemma \ref{structure},
the pushforward homomorphism $H_2(S,\Z)\arrow H_2(Y,\Z)$
is surjective.

We assume that the Hilbert scheme $\cal H$ of smooth surfaces
in $Y$ in the homology class of $S$
is smooth, which holds if $H$ is sufficiently ample.
We are free to replace $H$ by a large multiple
in the course of the argument.

The following lemma was suggested by Schoen's argument
on the integral Tate conjecture \cite[Theorem 0.5]{Schoen}.

\begin{lemma}
\label{algebraic}
Let $Y$ be a terminal projective complex 3-fold.
Write $S_{t_0}$ for the surface in $Y$ corresponding
to a point $t_0$ in $\cal H$,
with inclusion $i\colon S_{t_0}\arrow Y$. Write
$H_2(S_{t_0},\Z)_{\van}=\ker(i_*\colon H_2(S_{t_0},\Z)\arrow
H_2(Y,\Z))$. By Poincar\'e duality, identify $H^2(S_{t_0},\Z)$
with $H_2(S_{t_0},\Z)$.
Let $C$ be a nonempty open cone in $H^2(S_{t_0},\R)_{\van}$.
Suppose that there is a contractible open neighborhood $U$
of $t_0$ in $\cal H$
such that every element of $H^2(S_{t_0},\Z)_{\van}\cap C$
becomes a Hodge class on $S_t$ for some $t$ in $U$.
Then every element of $H_2(Y,\Z)$
whose image in $H_2(Y,\C)$ is in $H_{1,1}(Y)$ is algebraic.
\end{lemma}

\begin{proof}
By the proof of Lemma \ref{structure},
the pushforward $H_2(S_{t_0},\Z)\arrow H_2(Y,\Z)$ is surjective,
and so the pullback $H^2(Y,\Q)\arrow H^2(S_{t_0},\Q)$ is injective.
Therefore, the Hodge structure
on $H^2(S_{t_0},\Q)$ is pure of weight $2$. Still following
the proof of Lemma \ref{structure},
the polarization of $H^2(S_{t_0},\Q)$ by the intersection
form gives a canonical direct-sum decomposition of Hodge structures
$$H^2(S_{t_0},\Q)=H^2(Y,\Q)\oplus H^2(Y,\Q)^{\perp}.$$
In fact, this argument shows that
the surjection $i_*\colon H_2(S_{t},\Q)\arrow H_2(Y,\Q)$
is split as a map of variations of $\Q$-Hodge structures
over the space $\cal H$ of smooth surfaces $S$.
In particular,
any element of $H_2(Y,\Q)\cap H_{1,1}(Y)
\subset H_2(Y,\C)$ is the image of some element
in $H_2(S_{t_0},\Q)$ whose translate to every surface
$S_t$ is in $H_{1,1}(S_{t})$. Therefore, for any element
$\alpha$ of $H_2(Y,\Z)$ that maps into $H_{1,1}(Y)\subset H_2(Y,\C)$,
there is a positive integer $N$ and an element
$\beta$ of $H_2(S_{t_0},\Z)$ that lies in $H_{1,1}(S_{t})
\subset H_2(S_{t},\C)$ for every surface $S_t$
such that $i_*\beta=N\alpha$. 

Also, because $i_*\colon H_2(S_{t_0},\Z)\arrow H_2(X,\Z)$
is surjective, there is an element
$v\in H_2(S_{t_0},\Z)$ (not necessarily a Hodge class) with
$i_*v=\alpha$.

Let $u_0=\beta-Nv$ in $H_2(S_{t_0},\Z)$. Then $i_*u_0=0$;
that is, $u_0$ is in $H_2(S_{t_0},\Z)_{\van}$. Let
$T=u_0+N\cdot H_2(S_{t_0},\Z)_{\van}\subset H_2(S_{t_0},\Z)_{\van}$.
Since $T$ is a translate of a subgroup of finite index in
$H_2(S_{t_0},\Z)_{\van}$, $T$ has nonempty intersection with the open
cone $C$ in $H_2(S_{t_0},\R)_{\van}$.
Let $u$ be an element of $C\cap T$. Because $U$ is contractible,
we can canonically identify $H_2(S_t,\Z)$ with
$H_2(S_{t_0},\Z)$ for all $t$ in $U$. By our assumption on $C$,
$u$ becomes a Hodge class
on $H_2(S_t,\Z)$ for some $t$ in $U$.
By definition of $T$, we can write $u=u_0+Nw$ for some
$w$ in $H_2(S_{t_0},\Z)_{\van}$. We know that $\beta$ in $H_2(S_{t_0},\Z)$
is a Hodge class in $H^2(S_t,\Z)$ for all nearby surfaces
$S_t$. Since $u$ becomes a Hodge class
in $H^2(S_t,\Z)$, $\beta-u$ is a Hodge class in $H_2(S_t,\Z)$,
and $\beta-u=\beta-(u_0+Nw)=\beta-(\beta-Nv+Nw)=N(v-w)$.
So $v-w$ is a Hodge class in $H_2(S_t,\Z)$. By the Lefschetz
$(1,1)$ theorem, $v-w$ is algebraic on $S_t$.
And we have
$i_*(v-w)=i_*v=\alpha$. So $\alpha$ in $H_2(Y,\Z)$
is algebraic.
\end{proof}

We will prove the hypothesis of Lemma \ref{algebraic}
as Proposition \ref{hodge}. Given that, we now
finish the proof of Theorem \ref{minimal}.

Let $u$ be an element
of $H_2(X,\Z)\cap H_{1,1}(X)$. Topologically, $Y$ is obtained
from $Y$ by identifying the fibers $E_1,\ldots,E_r$
over singular points of $Y$ to points. So we have an exact
sequence
$$H_3(Y,\Z)\arrow \oplus_i H_2(E_i,\Z)\arrow H_2(X,\Z)
\arrow H_2(Y,\Z)\arrow \oplus_i H_1(E_i,\Z).$$
By Lemma \ref{sing}, $H_2(E_i,\Z)$ is spanned by algebraic curves
on $E_i$, for each $i$. The image of $u$ in $H_2(Y,\Z)$
is in $H_{1,1}(Y)$, and hence is in the image
of the Chow group $CH_1(Y)$ by Lemma \ref{algebraic}
and Proposition \ref{hodge}. (We use here that the integral
Hodge conjecture holds on every smooth projective surface,
by the Lefschetz $(1,1)$ theorem.)
Since $X\arrow Y$ is an isomorphism
outside a 0-dimensional subset of $Y$, it is clear
that $CH_1(X)\arrow CH_1(Y)$ is surjective.
Therefore there is a 1-cycle $\alpha$ on $X$ whose
image in $H_2(X,\Z)$ has the same image in $H_2(Y,\Z)$ as
$u$ does. By the exact sequence above, $\alpha-u$
in $H_2(X,\Z)$ is the image of some element of $\oplus H_2(E_i,\Z)$.
But $\oplus H_2(E_i,\Z)$ is spanned by algebraic cycles on $\cup_i E_i$
by Lemma \ref{sing}. Therefore $u$ is algebraic.
\end{proof}

\section{The variation of Hodge structure
associated to a family of surfaces}

To complete the proof of Theorem \ref{minimal},
we need to show that the variation of Hodge structures
on the family of surfaces in the 3-fold is as nontrivial
as possible (Proposition \ref{hodge}).
The first step is to rephrase the conclusion we want
in terms of a cup product on a general surface in the family
(Corollary \ref{surjective}), generalizing
Proposition 1 in Voisin \cite{Voisinint}.

Let $Y$ be a terminal complex projective 3-fold.
(We will eventually assume that $K_Y$ is trivial,
but it seems clearer to formulate the basic arguments
in greater generality.)
Let $H$
be a very ample line bundle on $Y$, and $S$ a smooth
surface in the linear system $|H|$. (It follows that $S$ is contained
in the smooth locus of $Y$.) We are free to replace $H$ by a large multiple
in the course of the argument.

By the proof of Lemma \ref{structure}, the Hodge structure
on $H^2(Y,\Q)$ is pure of weight 2, and the graded pieces
of the Hodge filtration on $H^2(Y,\C)$ are $H^2(Y,O)$,
$H^1(Y,\o^1_Y)$, and $H^0(Y,\o^2_Y)$.

Define the vanishing cohomology $H^2(S,\Z)_{\van}$ to be the kernel
of the pushforward homomorphism
$i_*\colon H^2(S,\Z)\cong H_2(S,\Z)\arrow H_2(Y,\Z)$. Likewise, write
\begin{align*}
H^2(S,O)_{\van} &= \ker(H^2(S,O)\arrow H^0(Y,\o^2_Y)^*)\\
H^1(S,\Omega^1_S)_{\van} &= \ker(H^1(S,\Omega^1)\arrow H^1(Y,\o^1_Y)^*)\\
H^0(S,\Omega^2_S)_{\van} &= \ker(H^0(S,\Omega^2)\arrow H^2(Y,\o^0_Y)^*).
\end{align*}
These maps are Hodge-graded pieces
of the surjection $H_2(S,\C)\arrow H_2(Y,\C)$ (dual to the pullback
$H^2(Y,\C)\arrow H^2(S,\C)$), and so they are also
surjective. By Serre duality, $H^2(Y,\o^0_Y)^*=H^2(Y,O)^*
\cong H^1(Y,K_Y)$. So we can also describe $H^0(S,\Omega^2_S)$
as the kernel of the pushforward $H^0(S,K_S)\arrow H^1(Y,K_Y)$.

The cohomology sheaves of $\o^j_Y$ are
in degrees $\geq 0$, and the 0th cohomology sheaf
is the sheaf $\Omega^{[j]}_Y$ of reflexive differentials,
because $Y$ is klt \cite[Theorems 5.4
and 7.12]{HJ}.
Also, $\o^j_Y$ is concentrated in degrees from 0 to $3-j$,
since $Y$ has dimension 3
\cite[Th\'eor\`eme V.6.2]{Guillen}. It follows that
$\o^3_Y$ is the canonical sheaf $K_Y$.

We assume that the Hilbert scheme $\cal H$ of smooth surfaces
in $Y$ in the homology class of $S$
is smooth, which holds if $H$ is sufficiently ample.
Then $H^0(S,N_{S/Y})$ is the tangent space to $\cal H$ at $S$.
Let $\delta$ be the class of the extension $0\arrow TS
\arrow TY|_S\arrow N_{S/Y}\arrow 0$ in $H^1(S,N_{S/Y}^*\otimes
TS)$. 
Then the product with $\delta$ is the Kodaira-Spencer
map $H^0(S,N_{S/Y})\arrow H^1(S,TS)$,
which describes how the isomorphism class of $S$ changes
as $S$ moves in $Y$.

For $u$ in $H^1(S,TS)$, the product with $u$
is a linear map
$$u\cdot \colon H^1(S,\Omega^1)\arrow H^2(S,O),$$
The dual map
$$u\cdot \colon H^0(S,K_S)\arrow H^1(S,\Omega^1)$$
can also be described as the product with $u$.
For $\lambda$ in $H^1(S,\Omega^1)$, define
$$\mu_{\lambda}\colon H^0(S,N_{S/Y})\arrow H^2(S,O)$$
by $\mu_{\lambda}(n)=(\delta n) \lambda$. This map
describes the failure of $\lambda\in H^2(S,\C)$ to remain
a $(1,1)$ class when the surface $S$ is deformed in $X$.
For $\lambda$ in $H^1(S,\Omega^1)_{\van}$, the map $\mu_{\lambda}$
lands in the vanishing subspace $H^2(S,O)_{\van}$,
because $H^2(S,\Q)=H^2(Y,\Q)\oplus H^2(S,\Q)_{\van}$
as Hodge structures, where
the Hodge structure on $H^2(Y,\Q)$ is unchanged as $S$ is deformed.

\begin{corollary}
\label{surjective}
Let $Y$ be a terminal projective complex 3-fold. Let $H$
be a very ample line bundle on $Y$, and $S$ a smooth
surface in the linear system $|H|$. Suppose that there
is an element $\lambda$ in $H^1(S,\Omega^1)_{\van}$
such that the linear map
$$\mu_{\lambda}\colon H^0(S,N_{S/Y})\arrow H^2(S,O)_{\van}$$
is surjective. Then there is a nonempty open cone $C$
in $H^2(S_{t_0},\R)_{\van}$ and
a contractible open neighborhood $U$
of $t_0$ in $\cal H$
such that every element of $H^2(S_{t_0},\Z)_{\van}\cap C$
becomes a Hodge class on $S_t$ for some $t$ in $U$.
\end{corollary}

For $Y$ smooth and $H^2(Y,O)=0$
(which implies that every element of $H_2(Y,\Z)$ is a Hodge class),
this was proved by Voisin \cite[Proposition 1]{Voisinint}.

\begin{proof}
The groups $H^2(S_t,\Z)$ form a weight-2 variation of Hodge
structures on the Hilbert scheme $\cal H$ of smooth surfaces $S_t\subset Y$.
Let $U$ be a contractible open neighborhood
of the given point $t_0$ in $\cal H$. We can canonically
identify $H^2(S_t,\C)$ with
$H^2(S_{t_0},\C)$ for all $t\in U$.
The surjectivity of $\mu_{\lambda}$ implies that the map from
$\cup_{t\in U}H^{1,1}(S_t,\R)_{\van}$
to $H^2(S_{t_0},\R)_{\van}$ is a submersion
at $\lambda$, I claim.

To prove that this map is a submersion, we
follow the argument of \cite[Proposition 1]{Voisinint},
modified so as not to assume that $H^2(Y,O)$ is zero.
Let $\pi\colon S_U\arrow U$ be the universal
<family of surfaces $S_t$, restricted to $t\in U$. Write $H^2_{\van}$ for 
the total space of the vector bundle $(R^2\pi_*\C)_{\van}$ over $U$,
with fibers $H^2(S_t,\C)_{\van}$. The Gauss-Manin connection gives a
trivialization of this bundle, and hence a projection map from
the total space to one fiber,
$\tau\colon H^2_{\van}\arrow H^2(S_{t_0},\C)_{\van}$.
Let $F^1H^2_{\van}$ be the submanifold
of $H^2$ whose fiber over each point $t\in U$ is the Hodge filtration
$$F^1H^2(S_t,\C)_{\van}=H^{2,0}(S_t)_{\van}\oplus H^{1,1}(S_t)_{\van}
\subset H^2(S_t,\C)_{\van}.$$
Let $\tau_1\colon
F^1H^2_{\van}\arrow H^2(S_{t_0},\C)_{\van}$ be the restriction of $\tau$
to $F^1H^2_{\van}$.
Let $\lambda$ be an element of $H^1(S_{t_0},\Omega^1)_{\van}$,
and $\widetilde{\lambda}$ any lift of $\lambda$
to $F^1H^2(S_{t_0},\C)_{\van}$. By the proof of Voisin
\cite[Lemma 2]{Voisinint} (modified since we are allowing $H^2(Y,O)$
to be nonzero), we have
the following equivalence:

\begin{lemma}
\label{submersion}
The map
$$\mu_{\lambda}\colon H^0(S_{t_0},N_{S_{t_0}/Y})\arrow H^2(S_{t_0},O)_{\van}$$
is surjective if and only if $\tau_1$ is a submersion
at $\widetilde{\lambda}$.
\end{lemma}

To relate Lemma \ref{submersion}
to cohomology with real coefficients, note
that surjectivity of $\mu_{\lambda}$ is a Zariski open 
condition on $\lambda$ in $H^1(S_{t_0},\Omega^1)_{\van}$. The vector space
$H^1(S_{t_0},\Omega^1)_{\van}$ has a real structure, given by
$$H^{1,1}(S_{t_0})_{\R,\van}=H^1(S_{t_0},\Omega^1)_{\van}\cap
H^2(S_{t_0},\R)_{\van}
\subset H^2(S_{t_0},\C)_{\van}.$$
Since we assume in this Corollary that
$\mu_{\lambda}$ is surjective for one $\lambda$ in
$H^1(S_{t_0},\Omega^1)_{\van}$, it is surjective for some
$\lambda$ in $H^{1,1}(S_{t_0})_{\R,\van}$.

In Lemma \ref{submersion}, take the lifting $\widetilde{\lambda}$
to be $\lambda$ itself. Then $\widetilde{\lambda}$ is real,
and so is $\tau_1(\widetilde{\lambda})$. By our assumption
on $\lambda$, Lemma \ref{submersion} gives that
$\tau_1$ is a submersion at $\widetilde{\lambda}$, and so the restriction
$$\tau_{1,\R}\colon H^{1,1}_{\R,\van}\arrow H^2(S_{t_0},\R)_{\van}$$
of $\tau_1$ to $\tau_1^{-1}(H^2(S_{t_0},\R)_{\van})$ is also
a submersion. Here $\tau_1^{-1}(H^2(S_{t_0},\R)_{\van})$ is identified
with
$$\cup_{t\in U}F^1H^2(S_t,\C)_{\van}\cap H^2(S_t,\R)_{\van}=\cup_{t\in U}
H^{1,1}(S_t,\R)_{\van}=: H^{1,1}_{\R,\van}.$$
Since $\tau_{1,\R}$ is a submersion at $\widetilde{\lambda}$
on the real manifold $H^{1,1}_{\R,\van}$ (a real vector bundle over $U$),
the image of $\tau_{1,\R}$ contains a nonempty open subset
of $H^2(S_{t_0},\R)_{\van}$, as we wanted.

The image of $\tau_{1,\R}$ is a cone, and so it
contains an open cone $C$
in $H^2(S_{t_0},\R)_{\van}$. Therefore,
all elements of $H^2(S_{t_0},\Z)_{\van}$
in the open cone $C$ become Hodge classes on $S_t$ for some $t$ in $U$.
Corollary \ref{surjective} is proved.
\end{proof} 

\begin{proposition}
\label{hodge}
Let $Y$ be a terminal projective complex 3-fold
with trivial canonical bundle.
Write $S_{t_0}$ for the surface in $Y$ corresponding
to a point $t_0$ in $\cal H$,
with inclusion $i\colon S_{t_0}\arrow Y$. Write
$H_2(S_{t_0},\Z)_{\van}=\ker(i_*\colon H_2(S_{t_0},\Z)\arrow
H_2(Y,\Z))$. Then there is a nonempty open cone $C$
in $H^2(S_{t_0},\R)_{\van}$ and
a contractible open neighborhood $U$
of $t_0$ in $\cal H$
such that every element of $H^2(S_{t_0},\Z)_{\van}\cap C$
becomes a Hodge class on $S_t$ for some $t$ in $U$.
\end{proposition}

\begin{proof}
Let $H$ be a very ample line bundle on $Y$ and let $S$ be a smooth
surface in $|nH|$ for a positive integer $n$. (We will eventually
take $n$ big enough and $S$ general in $|nH|$.)
Let $V=H^0(S,K_S)_{\van}$ and $V'=H^0(Y,O(S))/H^0(Y,O)$.
By the exact sequence of sheaves on $Y$
$$0\arrow O_Y\arrow O(S)\arrow O(S)|_S\arrow 0,$$
we can view $V'$ as a subspace of $H^0(S,O(S)|_S)=H^0(S,N_{S/Y})$.
For $n$ sufficiently large, the long exact sequence of cohomology gives
an exact sequence
$$0\arrow V'\arrow H^0(S,N_{S/Y})\arrow H^1(Y,O)\arrow 0.$$
Likewise, by definition of $V$, we have
an exact sequence
$$0\arrow V\arrow H^0(S,K_S)\arrow H^1(Y,K_Y)\arrow 0,$$
where the pushforward map shown is the boundary map
from the exact sequence of sheaves on $Y$:
$$0\arrow K_Y\arrow K_Y(S)\arrow K_S\arrow 0.$$

We will only need to move $S$ in its linear system
(although $H^1(Y,O)$ need not be zero). That is,
we will show that for a general $\lambda\in H^1(S,\Omega^1_S)_{\van}$
the restriction of $\mu_{\lambda}$ to $V'\subset H^0(S,N_{S/Y})$
maps onto $H^2(S,O)_{\van}=V^*$; by Corollary \ref{surjective},
that will finish the proof of Proposition \ref{hodge}.
We will see that these two vector spaces have the same dimension, using
that $K_Y$ is trivial, and so the argument just barely works.
(For $K_Y$ more positive, it would not work at all.)

Fix a trivialization of the canonical bundle $K_Y$.
This gives an isomorphism between the two short exact sequences
of sheaves above, in particular an isomorphism $K_S\cong N_{S/Y}$
of line bundles on $S$. So we have an isomorphism between
the two exact sequences of cohomology, including
an isomorphism $V\cong V'$. 
In terms of this 
identification, $\mu_{\lambda}$ for $\lambda$ in $H^1(S,\Omega^1)_{\van}$
is a linear map $V'\arrow V^*$. For varying $\lambda$,
this is equivalent to the pairing
\begin{align*}
\mu\colon V\times V'&\arrow H^1(S,\Omega^1)_{\van}\\
\mu(v,v')&=v(v'\delta),
\end{align*}
which is symmetric. (Recall that
$\delta$ is the class of the extension $0\arrow TS
\arrow TY|_S\arrow N_{S/Y}\arrow 0$ in $H^1(S,N_{S/Y}^*\otimes
TS)$.) (Proof: this pairing is the restriction of a symmetric
pairing $H^0(S,N_{S/Y})\otimes H^0(S,N_{S/Y})\arrow H^1(S,\Omega^1)$,
given by $u\otimes v\mapsto uv\gamma\delta$, where $u,v\in
H^0(S,N_{S/Y})$, and $\gamma\in H^0(S,N_{S/Y}^*\otimes (TS)^*
\otimes \Omega^1_S)$ is the natural map
$N_{S/Y}\otimes TS\cong \Omega^2_S\otimes TS\arrow \Omega^1_S$
of bundles on $S$.)
Serre duality $H^1(S,\Omega^1)_{\van}^*\cong H^1(S,\Omega^1)_{\van}$
gives a dual map
$$q=\mu^*\colon H^1(S,\Omega^1)_{\van}\arrow S^2V^*.$$
We can think of $q$ as a linear system of quadrics
in the projective space $P(V^*)$ of lines in $V$.
The condition that $\mu_{\lambda}$ from $V'\subset H^0(S,N_{S/Y})$
to $V^*=H^2(S,O)_{\van}$
is surjective for generic
$\lambda\in H^1(S,\Omega^1)_{\van}$
is equivalent to the condition that the quadric
defined by $q(\lambda)$ is smooth for generic $\lambda$.
Thus, by Corollary \ref{surjective}, Proposition \ref{hodge}
will follow if we can show
that the quadric $q(\lambda)$ is smooth for generic $\lambda$.

Note that we lose nothing by restricting the pairing
$\mu$ to the subspaces $V\subset H^0(S,K_S)$
and $V'\subset H^0(S,N_{S/Y})$. Indeed, as discussed
in the proof of Lemma \ref{structure}, the Hodge structure
$H^2(S,\Q)$ is polarized by the intersection form,
and the restriction $H^2(Y,\Q)\arrow H^2(S,\Q)$
is injective, with image a sub-Hodge structure. Therefore
$H^2(S,\Q)$ is the orthogonal direct sum of $H^2(Y,\Q)$
and its orthogonal complement. This gives a splitting
of each Hodge-graded piece of $H^2(S,\C)$. For example,
for $H^0(S,K_S)$, this gives the decomposition
$$H^0(S,K_S)\cong H^0(Y,\o^2_Y)\oplus H^0(S,K_S)_{\van}.$$

Therefore, we also have a canonical decomposition
of the isomorphic vector space $H^0(S,N_{S/Y})$.
This is the decomposition
$$H^0(S,N_{S/Y})\cong H^0(Y,TY)\oplus H^0(Y,O(S))/H^0(Y,O).$$
Thinking of $H^0(S,N_{S/Y})$ as the first-order
deformation space of $S$ in $Y$, these two subspaces
correspond to: moving $S$ by automorphisms of $Y$,
and moving $S$ in its linear system. The first type
of move does not change the Hodge structure of $S$,
and so it is irrelevant to our purpose (trying
to make a given integral cohomology class on $S$
into a Hodge class).

We use the following consequence of Bertini's theorem
from Voisin \cite[Lemma 15]{Voisinint}, which we apply
to our space $V$ (identified with $V'$)
and $W=H^1(S,\Omega^1)$. (Note that we follow the numbering
of statements in the published version of \cite{Voisinint},
not the preprint.)

\begin{lemma}
\label{symmetric}
Let $\mu\colon V\otimes V\arrow W$ be symmetric
and let $q\colon W^*\arrow S^2V^*$ be its dual.
For $v$ in $V$, write $\mu_v\colon V\arrow W$
for the corresponding linear map.
Think of $q$ as a linear system of quadrics in $\P(V^*)$.
Then the generic quadric in $\im(q)$ is smooth
if the following condition holds. There is no closed subvariety
$Z\subset \P(V^*)$
contained in the base locus of $\im(q)$
and satisfying:
$$\rank(\mu_v)\leq \dim(Z)$$
for all $v\in Z$.
\end{lemma}

We have to show that such a subvariety $Z$ does not exist
for a general surface $S\in |nH|$ with $n$ sufficiently divisible.
We follow the outline of H\"oring and Voisin's argument
(\cite[after Lemma 3.35]{HV}, extending \cite[after Lemma 7]{Voisinint}
in the smooth case). After replacing the very ample
line bundle $H$ by a multiple if necessary, we can assume that
$H^i(Y,O(lH))=0$ for $i>0$ and $l>0$.
We degenerate the general
surface $S$ to a surface with many nodes, as follows.
Consider a general symmetric $n\times n$ matrix $A$
with entries in $H^0(Y,O(H))$. Let $S_0$ be the surface in $Y$ defined
by the determinant of $A$ in $H^0(Y,O(nH))$. By the assumption
of generality, $S_0$ is contained in the smooth locus
of $Y$. By Barth \cite{Barth},
the singular set of $S_0$ consists of 
$N$ nodes, where
$$N=\binom{n+1}{3}H^3.$$

Let ${\cal S}\arrow \Delta$ be a Lefschetz degeneration
of surfaces
$S_t\in |nH|$ over the unit disc $\Delta$ such that
the central fiber $S_0$ has nodes $x_1,\ldots,x_N$
as singularities. The map $q$ above makes sense
for any smooth surface $S$ in a 3-fold.
Voisin showed that the limiting space
$$\lim_{t\arrow 0} \im(q_t\colon H^1(S_t,\Omega^1_{S_t})\arrow
(H^0(S_t,K_{S_t})\otimes H^0(S_t,O(nH)))^*),$$
which is a linear subspace of $(H^0(S_0,K_{S_0})\otimes H^0(S_0,O(nH)))^*$,
contains for each $1\leq i\leq N$ the multiplication-evaluation
map which is the composite
$$H^0(S_0,K_{S_0})\otimes H^0(S_0,O(nH))\arrow H^0(S_0,K_{S_0}(nH))
\arrow K_{S_0}(nH)|_{x_i}$$
\cite[Lemma 7]{Voisinint}.

Recall that we have identified $V=H^0(S,K_S)_{\van}$ with $V'=
H^0(Y,O(S))/H^0(Y,O)$.
When we degenerate a general surface $S$ to the nodal
surface $S_0$, the base locus $B\subset P(V^*)$ of $\im(q)$ specializes
to a subspace of the base locus $B_0$ of $\im(q_0)\subset P(V_0^*)$,
where 
$$V_0=H^0(S_0,K_{S_0})\cong V'_0=H^0(S_0,O_{S_0}(nH)).$$
Let $W$ be the set of nodes of $S_0$.
By Voisin's lemma
just mentioned, $B_0$ is contained in
$$C_0:=\{ v\in P(V^*_0): v^2|_W=0\}.$$
As a set, $C_0$ is a linear subspace:
\begin{align*}
C_0 &=\{ v\in P(V^*_0):v|_W=0\}\\
&= P(H^0(S_0,K_{S_0}\otimes I_W)^*).
\end{align*}
By \cite[eq.\ 3.36]{HV}, extending \cite[Corollary 3]{Voisinint}
in the smooth case (using only that $H^i(Y,O(lH))=0$ for $i>0$ and $l>0$),
we have $h^0(Y,K_Y(nH)\otimes I_W)\leq cn^2$
for some constant $c$ independent of $n$. Thus the base locus
$B_0$ of $\im(q_0)$ has dimension at most
$cn^2$, for some constant $c$ independent of $n$. By specializing,
the base locus $B$ of $\im(q)$ also has dimension at most
$cn^2$ for general surfaces $S$ in $|nH|$.

By our assumption on the subvariety $Z$ of $B$,
for $v\in Z$ we have
\begin{align*}
\rank(\mu_v\colon V\arrow H^1(S,\Omega^1))&\leq \dim(Z)\\
&\leq cn^2.
\end{align*}
By the following lemma,
it follows that $\dim(Z)\leq A$ for some constant $A$
independent of $n$. This is H\"oring-Voisin's \cite[Lemma 3.37]{HV},
extending \cite[Lemma 12]{Voisinint} in the smooth case.
(As before, we follow the numbering from the published version
of \cite{Voisinint}.
In our case, $H^2(Y,O)$ need not be zero, but that is not used
in these proofs.)

\begin{lemma}
Let $Y$ be a Gorenstein projective 3-fold with isolated
canonical singularities, $H$ as above.
For each positive integer $n$, let $S$ be a general surface
in $|nH|$, and define $V,V',\mu$ associated to $S$ as above.
Let $c$ be any positive constant. Then there is a constant
$A$ such that the sets
$$\Gamma=\{v\in V: \rank(\mu_v)\leq cn^2\}$$
and
$$\Gamma'=\{v'\in V': \rank(\mu_{v'})\leq cn^2\}$$
both have dimension bounded by $A$, independent of $n$.
\end{lemma}

By our assumption on the subvariety $Z$ of $B$ again, for $v\in Z$
we have
\begin{align*}
\rank(\mu_v\colon V\arrow H^1(S,\Omega^1))&\leq \dim(Z)\\
&\leq A.
\end{align*}
This implies that $Z$ is empty by Lemma \ref{rank}, to be proved
next.
But $Z$ is a variety, so we have a contradiction.
This completes the proof that
the generic quadric in the linear system $\im(q)$ is smooth.
Proposition \ref{hodge} is proved.
\end{proof}

To complete the proof of Proposition \ref{hodge} and hence
Theorem \ref{minimal}, it remains to prove the following lemma.

\begin{lemma}
\label{rank}
Let $Y$ be a terminal projective 3-fold with $K_Y$ trivial,
$H$ as above.
Let $A$ be a positive integer. Let $S\in |nH|$ be general,
with $n$ large enough (depending on $A$). Let $V=H^0(S,K_S)_{\van}$
and $\mu_v\colon V'\arrow H^1(S,\Omega^1)$ the product with an
element $v\in V$,
as defined above. Then the set
$$W=\{v\in V:\rank(\mu_v)<A\}$$
is equal to $0$.
\end{lemma}

\begin{proof}
We have to modify the proof of Voisin's Lemma 13 \cite{Voisinint}
to allow $Y$ to be singular and also to have $H^2(Y,O)$ not zero.
We use H\"oring and Voisin's ideas on how to deal with $Y$ being
singular, by working on the smooth surface $S$
as far as possible \cite[proof of Proposition 3.22]{HV}.

Let $S$ be a smooth surface in $|nH|$.
Consider the following exact sequences of vector bundles on $S$,
constructed from the normal bundle sequence of $S$ in $Y$:
$$0\arrow \Omega^1_S(nH)\arrow \Omega^2_Y|_S(2nH)\arrow K_S(2nH)\arrow 0$$
and
$$0\arrow O_S\arrow \Omega^1_Y|_S(nH)\arrow \Omega^1_S(nH)\arrow 0.$$
Let $\delta_1$ and $\delta_2$ be the resulting boundary maps:
$$\delta_1\colon H^0(S,K_S(2nH))\arrow H^1(S,\Omega^1_S(nH))$$
and
$$\delta_2\colon H^1(S,\Omega^1_S(nH))\arrow H^2(S,O).$$
Let $\delta=\delta_2\circ \delta_1\colon H^0(S,K_S(2nH))
\arrow H^2(S,O)$.

\begin{lemma}
\label{imdelta}
The image of $\delta$ is $H^2(S,O)_{\van}$,
for large enough $n$
and any $S$ as above.
\end{lemma}

\begin{proof}
We first show that $\delta_1$ is surjective.
By the long exact sequence of cohomology
associated to the first exact sequence above, it suffices
to show that $H^1(S,\Omega^2_Y|_S(2nH))$ is zero. In terms
of the sheaf $\Omega^{[2]}_Y$ of reflexive differentials,
we have an exact sequence of sheaves on $Y$:
$$0\arrow \Omega^{[2]}_Y(nH)\arrow \Omega^{[2]}_Y(2nH)
\arrow \Omega^2_Y|_S(2nH)\arrow 0.$$
By Serre vanishing on $Y$, both $H^1(Y,\Omega^{[2]}_Y(2nH))$
and $H^2(Y,\Omega^{[2]}_Y(nH))$ vanish for large $n$,
and so $H^1(S,\Omega^2_Y|_S(2nH))=0$ for all smooth $S$
in $|nH|$ with $n$ large.

Next, the long exact sequence involving $\delta_2$ shows
that the cokernel of $\delta_2$ is contained
in $H^2(S,\Omega^1_Y|_S(nH))$. Since $K_S=nH|_S$
by the adjunction formula, the dual of that $H^2$ space
is $H^0(S,TY|_S)=H^0(S,\Omega^2_Y|_S)$. 

By the exact sequence
$$0\arrow \Omega^{[2]}_Y(-nH)\arrow \Omega^{[2]}_Y
\arrow \Omega^2_Y|_S\arrow 0$$
of sheaves on $Y$,
we have an exact sequence
$$H^0(X,\Omega^{[2]}_Y(-nH))\arrow H^0(Y,\Omega^{[2]}_X)
\arrow H^0(S,\Omega^{[2]}_Y|_S)\arrow H^1(Y,\Omega^{[2]}_Y(-nH)).$$
Since $Y$ is normal, the sheaf $\Omega^{[2]}_Y$ is reflexive,
and $\dim(Y)>1$, the groups on the left and right are zero
for $n$ large. (Consider an embedding of $Y$ into some $\P^N$
and use Serre vanishing and Serre duality on $\P^N$, as in
\cite[proof of Corollary III.7.8]{Hartshorne}.) So
the map $H^0(X,\Omega^{[2]})\arrow H^0(S,\Omega^2_Y|_S)$
is an isomorphism. By the results above on $\delta_1$
and $\delta_2$, this gives an exact sequence
$$H^0(S,K_S(2nH))\xrightarrow[\delta]{} H^2(S,O)
\arrow H^0(Y,\Omega^{[2]})^*.$$

Finally, we need to rephrase this in terms of du Bois's
object $\o^2_Y$ in the derived category of $Y$.
The cohomology sheaves of $\o^j_Y$ are
in degrees $\geq 0$, and the 0th cohomology sheaf
is $\Omega^{[j]}_Y$ because $Y$ is klt \cite[Theorems 5.4
and 7.12]{HJ}. 
So there is a natural map $\Omega^{[2]}_Y\arrow
\o^2_Y$ in $D(Y)$. Because the other cohomology sheaves
of $\o^2_Y$ are in degrees $>0$, it is immediate
that the map $H^0(Y,\Omega^{[2]})\arrow H^0(Y,\o^2_Y)$
is an isomorphism. So the previous paragraph yields
an exact sequence:
$$H^0(S,K_S(2nH))\xrightarrow[\delta]{} H^2(S,O)
\arrow H^0(Y,\o^2_Y)^*.$$
Equivalently, the image of $\delta$ is $H^2(S,O)_{\van}$.
\end{proof}

Assume that $v\in V$
satisfies the condition that $\rank(\mu_v)<A$.
Using that $n$ is sufficiently large, H\"oring and Voisin show
that $\delta(H^0(S,O_S(3nH)))$
is orthogonal to $v$ with respect to Serre duality
\cite[after Proposition 3.40]{HV}. By Lemma \ref{imdelta},
$H^2(S,O)_{\van}$ is orthogonal to $v$.
Since $V=H^0(S,K_S)_{\van}$ is dual to $H^2(S,O)_{\van}$, it follows
that $v=0$. Lemma \ref{rank} is proved. This also completes
the proofs of Proposition \ref{hodge} and Theorem \ref{minimal}.
\end{proof}

\section{The integral Tate conjecture for 3-folds}

We now prove the integral Tate conjecture for 3-folds
in characteristic zero that are rationally connected
or have Kodaira dimension zero
with $h^0(X,K_X)>0$ (Theorem \ref{tatezero}).
In any characteristic, we will prove the integral Tate conjecture
for abelian 3-folds (Theorem \ref{abeliantate}).

\begin{theorem}
\label{tatezero}
Let $X$ be a smooth projective 3-fold
over the algebraic closure of a finitely generated
field of characteristic zero. 
If $X$ is rationally connected or it has Kodaira dimension zero
with $h^0(X,K_X)>0$ (hence equal to 1), then $X$ satisfies
the integral Tate conjecture.
\end{theorem}

\begin{proof}
We start by proving the following known lemma.

\begin{lemma}
\label{codim1}
Let $X$ be a smooth projective variety over the separable
closure $k_s$ of a finitely generated field $k$. For codimension-1 cycles
on $X$, the Tate conjecture implies the integral Tate conjecture.
\end{lemma}

\begin{proof}
For a prime number $l$ invertible in $k$, the Kummer sequence 
$$0\arrow \mu_{l^r}\arrow G_m\xrightarrow[l^r]{} G_m\arrow 0$$
of \'etale sheaves on $X$ gives a long exact sequence of cohomology,
and hence an exact sequence involving the Picard and Brauer groups:
$$0\arrow \Pic(X)/l^r\arrow H^2_{\et}(X,\mu_{l^r})\arrow \Hom(\Z/l^r,
\Br(X))\arrow 0.$$
Writing $NS(X)$ for the group of divisors modulo algebraic
equivalence, we have $\Pic(X)/l^r=NS(X)/l^r$, because the group
of $k_s$-points of an abelian variety
is $l$-divisible. Since $NS(X)$ is finitely
generated, taking inverse limits gives an exact sequence:
$$0\arrow NS(X)\otimes\Z_l\arrow H^2(X,\Z_l(1))\arrow \Hom(\Q_l/\Z_l,
\Br(X))\arrow 0.$$
The last group is automatically torsion-free. It follows that
the Tate conjecture implies the integral
Tate conjecture in the case of codimension-1 cycles.
\end{proof}

\begin{lemma}
\label{tatelemma}
Let $X$ be a smooth projective 3-fold over the algebraic closure
$\overline{k}$ of a finitely generated field $k$ of characteristic
zero. Suppose that the Tate conjecture holds for codimension-1
cycles on $X$, and that the integral Hodge conjecture holds
on $X_{\C}$ for some embedding $\overline{k}\inj \C$.
Then the integral Tate conjecture holds for $X$ (over $\overline{k}$).
\end{lemma}

\begin{proof}
By Lemma \ref{codim1}, the integral Tate conjecture
holds for codimension-1 cycles on $X$.
It remains to prove integral Tate for 1-cycles on $X$.
Let $u\in H^4(X,\Z_l(2))$ be a Tate class; that is, $u$
is fixed by $\Gal(\overline{k}/l)$ for some finite extension
$l$ of $k$. Let $H$ be an ample line bundle on $X$.
Multiplication by the class of $H$ is an isomorphism
from $H^2(X,\Q_l(1))$ to $H^4(X,\Q_l(2))$, by the hard
Lefschetz theorem. So there is a positive integer $N$ with
$Nu=Hv$ for some $v\in H^2(X,\Z_l(1))$. Because the isomorphism
from $H^2(X,\Q_l(1))$ to $H^4(X,\Q_l(2))$ is Galois-equivariant,
$v$ is a Tate class (this works even if there is torsion
in $H^2(X,\Z_l(1))$, because we are considering
Tate classes over $\overline{k}$). By our assumptions,
$v$ is algebraic, that is, a $\Z_l$-linear combination
of classes of subvarieties of $X$. So $Hv=Nu$ is algebraic and thus
a $\Z_l$-linear combination of classes of curves on $X$.

In particular, $Nu$ is a $\Z_l$-linear combination of Hodge classes
in $H^4(X_{\C},\Z)$. Since the subgroup of Hodge classes is a summand
in $H^4(X_{\C},\Z)$, it follows that $u$ is a $\Z_l$-linear
combination of Hodge classes in $H^4(X_{\C},\Z)$. Since
the integral Hodge conjecture holds for $X_{\C}$, $u$
is a $\Z_l$-linear combination of classes of curves on $X$.
\end{proof}

We prove Theorem \ref{tatezero} using Lemma \ref{tatelemma}.
The integral Hodge conjecture holds for rationally connected 3-folds
by Voisin \cite[Theorem 2]{Voisinint} and for 3-folds $X$ with
Kodaira dimension zero and $h^0(X,K_X)=1$
by Theorem \ref{minimal}, generalizing
Voisin \cite[Theorem 2]{Voisinint}. It remains to check the
Tate conjecture in codimension 1 for $X$
over $\overline{k}$. That is clear if
$h^{0,2}(X)=0$; then all of $H^2(X_{\C},\Z)$ is algebraic by the
Lefschetz $(1,1)$ theorem, and so all of $H^2(X,\Z_l(1))$ is
algebraic. That covers the case where $X$ is rationally connected.

It remains to prove the Tate conjecture in codimension 1
for a 3-fold $X$ over $\overline{k}$ of Kodaira dimension zero
with $h^{0,2}(X)>0$. Let $Y$ be a minimal model of $X$;
then $Y$ is terminal and has torsion canonical bundle.
By H\"oring and Peternell, generalizing the Beauville-Bogomolov
structure theorem
to singular varieties,
there is a projective variety $Z$
with canonical singularities and a finite morphism $Z\arrow Y$,
\'etale in codimension one, such that $Z$ is a product of an abelian
variety, (singular) irreducible symplectic varieties,
and (singular) Calabi-Yau varieties in a strict sense
\cite[Theorem 1.5]{HP}. Their theorem is stated over $\C$,
but that implies the statement over $\overline{k}$. H\"oring and Peternell
build on earlier work by Druel and Greb-Guenancia-Kebekus
\cite{Druel,GGK}.

Since $Y$ has dimension 3
and $h^0(Z,\Omega^{[2]})\geq h^0(Y,\Omega^{[2]})=h^0(X,\Omega^2)>0$,
the only possibilities are: $Z$ is an abelian 3-fold
or the product of an elliptic curve and a K3 surface with canonical
singularities. (A strict Calabi-Yau 3-fold $Z$ has $h^0(Z,\Omega^{[2]})=0$,
by definition.) So there is a resolution of singularities $Z_1$
of $Z$ which is either an abelian 3-fold or the product
of an elliptic curve and a smooth K3 surface.
Since we have a dominant rational map $Z_1\dashrightarrow X$,
the Tate conjecture in codimension 1 for $X$ will follow from
the same statement for $Z_1$ \cite[Theorem 5.2]{Tate94}.

It remains to prove the Tate conjecture in codimension 1
for $Z_1$, which is either an abelian 3-fold
or the product of a K3 surface 
and an elliptic curve over $\overline{k}$. 
Faltings proved
the Tate conjecture in codimension 1 for all abelian varieties
over number fields
\cite{Faltings}, extended to all finitely generated
fields of characteristic zero by Zarhin. Finally, the Tate
conjecture holds for K3 surfaces in
characteristic zero, by Tankeev \cite{Tankeev}.
Since $H^2(S\times E,\Q)\cong H^2(S,\Q)\oplus
H^2(E,\Q)$ for a K3 surface $S$ and an elliptic curve
$E$, the Tate conjecture in codimension 1 holds for $S\times E$.
\end{proof}

\section{The integral Tate conjecture for abelian 3-folds
in any characteristic}

We now prove the integral Tate conjecture for abelian 3-folds
in any characteristic. In characteristic zero, we have already
shown this in Theorem \ref{tatezero}.
However, it turns out that a more elementary proof works in any
characteristic,
modeled on Grabowski's proof of the integral Hodge conjecture
for complex abelian 3-folds \cite[Corollary 3.1.9]{Grabowski}.
More generally,
we show that the integral Tate conjecture holds for 1-cycles
on all abelian varieties of dimension $g$ if
the ``minimal class'' $\theta^{g-1}/(g-1)!$ is algebraic
on every principally
polarized abelian variety $(X,\theta)$ of dimension $g$
(Proposition \ref{minimalclass}).

\begin{theorem}
\label{abeliantate}
Let $X$ be an abelian 3-fold over the separable closure
of a finitely generated field. Then the integral Tate conjecture
holds for $X$.
\end{theorem}

\begin{proof}
The argument is based on Beauville's Fourier transform for Chow
groups of abelian varieties, inspired by
Mukai's Fourier transform for derived categories.
Write $CH^*(X)_{\Q}$ for $CH^*(X)\otimes\Q$.
Let $X$ be an abelian variety of dimension $g$ over a field $k$,
with dual abelian variety $\X:=\Pic^0(X)$,
and let $f\colon X\times \X\arrow X$ and
$g\colon X\times \X\arrow \X$ be the projections. The {\it Fourier
transform }$F_X\colon CH^*(X)_{\Q}\arrow CH^*(\X)_{\Q}$
is the linear map
$$F_X(u)=g_*(f^*(u)\cdot e^{c_1(L)}),$$
where $L$ is the Poincar\'e line bundle on $X\times \X$
and $e^{c_1(L)}=\sum_{j=0}^g c_1(L)^j/j!$. For $k$ separably
closed, define the Fourier
transform $H^*(X,\Q_l(*))\arrow H^*(\X,\Q_l(*))$
by the same formula.

By Beauville, the Fourier transform sends $H^j(X,\Z_l(a))$
to $H^{2g-j}(X,\Z_l(a+g-j)$, and this map is an isomorphism
\cite[Proposition 1]{BeauvilleFourier}. By contrast, it is not clear
whether the Fourier transform can be defined integrally on Chow groups;
that actually fails over a general field, by Esnault
\cite[section 3.1]{MP}. Beauville's proof (for complex abelian
varieties) uses that the integral cohomology of an abelian
variety is an exterior algebra
over $\Z$, and the same argument works for the $\Z_l$-cohomology
of an abelian variety over any separably closed field.

Next, let $\theta\in H^2(X,\Z_l(1))$ be the first Chern class
of a principal polarization on an abelian variety $X$.
Then we can identify $\X$ with $X$, and the Fourier transform satisfies
$$F_X(\theta^j/j!)=(-1)^{g-j}\, \theta^{g-j}/(g-j)!$$
\cite[Lemme 1]{BeauvilleFourier}.
Here $\theta^j/j!$ lies in $H^{2j}(X,\Z_l(j))$ (although it is not
obviously algebraic, meaning the class of an algebraic cycle
with $\Z_l$-coefficients). Finally, let $h\colon X\arrow Y$
be an isogeny, and write $\widehat{h}\colon \widehat{Y}\arrow
\widehat{X}$ for the dual isogeny. Then the Fourier transform switches
pullback and pushforward, in the sense that for $u\in CH^*(Y)_{\Q}$,
$$F_X(h^*(u))=\widehat{h}_*(F_Y(u))$$
\cite[Proposition 3(iii)]{BeauvilleFourier}.

The following is the analog for the integral Tate conjecture
of Grabowski's argument on the integral Hodge conjecture
\cite[Proposition 3.1.8]{Grabowski}.

\begin{proposition}
\label{minimalclass}
Let $k$ be the separable closure of a finitely generated field.
Suppose that for every principally polarized abelian variety
$(Y,\theta)$ of dimension $g$ over $k$, the minimal class
$\theta^{g-1}/(g-1)!\in H^{2g-2}(Y,\Z_l(g-1))$ is algebraic.
Then the integral Tate conjecture for 1-cycles
holds for all abelian varieties of dimension $g$ over $k$.
\end{proposition}

\begin{proof}
Let $X$ be an abelian variety of dimension $g$ over $k$,
and let $u$ be a Tate class in $H^2(X,\Z_l(1))$ (meaning
that $u$ is fixed by some open subgroup of the Galois group).
The Tate conjecture
holds for codimension-1 cycles on abelian varieties over $k$,
by Tate \cite{Tateendo}, Faltings \cite{Faltings},
and Zarhin. This implies the integral Tate conjecture
for codimension-1 cycles on $X$, by Lemma \ref{codim1}. So $u$
is a $\Z_l$-linear combination of classes of line bundles,
hence of ample line bundles.

For each ample line bundle $L$ on $X$, there is 
a principally polarized abelian variety $(Y,\theta)$
and an isogeny $h\colon X\arrow Y$ with $c_1(L)=h^*\theta$
\cite[Corollary 1, p.~234]{Mumford}. Then the Fourier transform
of $c_1(L)$ is given by:
\begin{align*}
F_X(c_1(L))&=F_X(h^*\theta)\\
&=\widehat{h}_*(F_Y(\theta))\\
&=\widehat{h}_*(\theta^{g-1}/(g-1)!).
\end{align*}
By assumption, $\theta^{g-1}/(g-1)!$ in $H^{2g-2}(Y,\Z_l(g-1))$
is algebraic (with $\Z_l$ coefficients).
Since the pushforward preserves algebraic classes,
the equality above shows that
$F_X(c_1(L))$ is algebraic. By the previous paragraph,
it follows that the Fourier transform
of any Tate class in $H^2(X,\Z_l(1))$
is algebraic in $H^{2g-2}(\X,\Z_l(g-1))$. 

Since the Fourier transform is Galois-equivariant and is an isomorphism
from $H^2(X,\Z_l(1))$ to $H^{2g-2}(\X,\Z_l(g-1))$, it sends
Tate classes bijectively to Tate classes. This proves the integral
Tate conjecture for 1-cycles on $\X$, hence for 1-cycles on every abelian
variety of dimension $g$ over $k$.
\end{proof}

We now return to the proof of Theorem \ref{abeliantate}.
Let $k$ be the separable closure of a finitely generated
field, and let $X$ be an abelian 3-fold over $k$.
We want to prove the integral Tate conjecture for $X$.

By Proposition \ref{minimalclass}, it suffices to show
that for every principally polarized abelian 3-fold $(X,\theta)$
over $k$,
the class $\theta^2/2$ in $H^4(X,\Z_l(2))$ is algebraic.
(This is clear for $l\neq 2$.) A general principally
polarized abelian 3-fold $X$ over $k$ is the Jacobian of a curve $C$
of genus 3. In that case, choosing a $k$-point of $C$
determines an embedding of $C$ into $X$, and the cohomology
class of $C$ on $X$ is $\theta^2/2$ by Poincar\'e's formula
\cite[p.~350]{GH}. (Poincar\'e proved this for Jacobian
varieties over $\C$, but that implies the same formula in $l$-adic
cohomology for Jacobians in any characteristic.)
By the specialization homomorphism on Chow
groups \cite[Proposition 2.6, Example 20.3.5]{Fultonintersection},
it follows that $\theta^2/2$
is algebraic for every principally polarized abelian 3-fold over $k$.
Theorem \ref{abeliantate} is proved.
\end{proof}


\small \sc UCLA Mathematics Department, Box 951555,
Los Angeles, CA 90095-1555

totaro@math.ucla.edu
\end{document}